\documentclass[12pt]{amsart}
\usepackage[active]{srcltx}
\usepackage{a4wide}
\usepackage{amsthm,amsfonts,amsmath,mathrsfs,amssymb}
\usepackage{dsfont}
\usepackage[T1]{fontenc}
\usepackage[utf8]{inputenc}\RequirePackage{amsmath}
\usepackage{fourier}\RequirePackage{amssymb}
\usepackage{pdfpages}\RequirePackage{epsfig}
\usepackage{graphicx}
\def\a{\alpha}       \def\b{\beta}        \def\g{\gamma}
       \def\la{\lambda}     \def\om{\omega}

\def\D{{\mathbb D}}  \def\T{{\mathbb T}}
\def\C{{\mathbb C}}  
\def\R{{\mathbb R}}

\def\({\left(}       \def\){\right)}

\newtheorem{prop}{\sc Proposition}

\newtheorem{thm}[prop]{\sc Theorem}

\begin{document}
\title[Univalent polynomials with critical points on the circle]{On  univalent polynomials with critical points on the unit circle}
\author[M.J. Mart\'{\i}n]{Mar\'{\i}a J. Mart\'{\i}n}
\address{University of Eastern Finland, Department of Physics and Mathematics, P.O. Box 111, 80101 Joensuu, Finland}
\email{maria.martin@uef.fi}
\author[D. Vukoti\'c]{Dragan Vukoti\'c}
\address{Departamento de Matem\'aticas, Universidad Aut\'onoma de
Madrid, 28049 Madrid, Spain} \email{dragan.vukotic@uam.es}
\urladdr{http://www.uam.es/dragan.vukotic}
\thanks{The authors are supported by MTM2015-65792-P from MINECO/FEDER and partially by the Thematic Research Network MTM2015-69323-REDT, MINECO, Spain.}
\date{25 June, 2017.}
\begin{abstract}
Brannan showed that a normalized univalent polynomial of the form $P(z)=z+a_2 z^2+\ldots + a_{n-1}z^{n-1}+\frac{z^n}{n}$ is starlike if and only if $a_2=\ldots=a_{n-1}=0$. We give a new and simple proof of his result, showing further that it is also equivalent to the membership of $P$ in the Noshiro-Warschawski class of univalent functions whose derivative has positive real part in the disk. Both proofs are based on the Fej\'er lemma for trigonometric polynomials with positive real part.
\end{abstract}
\maketitle
\section*{Introduction}
 \label{sec-intro}
\par
Let $\D$ denote the unit disk in the complex plane and $S$ the class of all normalized univalent (that is, analytic and one-to-one) functions in $\D$ with the Taylor series $f(z)=z+\sum_{n=2}^\infty a_n z^n$. This class and its several natural subclasses have been extensively studied in the literature \cite{Du}, \cite{P}.
\par
A basic step in understanding  the class $S$ is the study of the polynomials in the class. Dieudonn\'e \cite{Di} characterized the univalence of a polynomial in terms of the roots of an associated trigonometric equation, which is useful for applications but not explicit. Various authors have afterwards sought a simple explicit characterization of the coefficient regions of univalent polynomials, without involving any additional parameters. In such a study,  Dieudonn\'e's result is often used as a starting point; see \cite[\S~8.6]{Du}. However, this has so far only been done for the polynomials of very small degree; the case $n=4$ is already quite complicated.
\par
If a normalized polynomial $P(z)=z+a_2 z^2+\ldots + a_{n-1}z^{n-1}+a_n z^n$ is locally univalent in the unit disk then it is easy to see that $|a_n|\le \frac1{n}$ since its derivative cannot vanish in $\D$; to this end, just write $P^\prime (z)=n a_n \prod_{k=1}^{n-1} (z-\a_k)$. In this note we will be interested mainly in the polynomials with the maximum modulus of the leading coefficient: $|a_n|=\frac1{n}$. It is not difficult to see that such a polynomial has all of its critical points on the unit circle. Conversely, it is trivial to see that every polynomial with critical points on the unit circle and normalized so that $P^\prime (0)=1$ must satisfy the condition $|a_n|=\frac1{n}$. Since the polynomial $P$ shares many properties (including those we are interested in here) with any of its rotations: $P_\la(z)=\overline{\la} P(\la z)$, where $|\la|=1$, it suffices to  consider only the case when $a_n=1/n$.
\par
An analytic function in a convex domain whose derivative has positive real part is univalent, as expressed by the Noshiro-Warschaw\-ski criterion for univalence \cite[Theorem~2.16]{Du}; in view of this, the normalized class of all functions in $S$ with $\mathrm{Re} f^\prime (z) > 0$ in $\D$ is often called the Noshiro-Warschawski class. In this note we show that a normalized univalent polynomial with $|a_n|= \frac1{n}$ belongs to the Noshiro-Warschaw\-ski class if and only if $a_2=\ldots=a_{n-1}=0$. The key tool in the proof will be the classical Fej\'er lemma for trigonometric polynomials with positive real part on the circle.
\par
Another important subclass of $S$ is that of the starlike functions. A set $E$ is said to be \textit{starlike\/} with respect to the origin if for every $z\in E$ the segment $[0,z]$ is contained in $E$. A function $f$ is said to be starlike if it is a univalent function of the disk onto a domain starlike with respect to the origin. The usual notation for the subclass of $S$ consisting of all starlike functions is $S^*$. It is well-known \cite[Theorem~2.10]{Du} that an analytic function $f$ in $\D$, normalized so that $f(0)=f^\prime(0)-1=0$, belongs to $S^*$ if and only if $\mathrm{Re\,} \frac{z f^\prime(z)}{f(z)}>0$ for all $z\in\D$. For the normalized polynomials with $a_n=\frac1{n}$, Brannan \cite{Br} showed that a polynomial of this form is starlike if and only if $a_2=\ldots=a_{n-1}=0$ by relying on the criterion of Dieudonn\'e for univalence. We will give a different proof of Brannan's result, again by using Fej\'er's lemma instead of Dieudonn\'e's criterion.
\section{Preliminary facts on polynomials}
 \label{sec-prelim}
\par
\textbf{Some basic facts about polynomials}. Given a complex polynomial of degree $n$: $P(z)=\sum_{k=0}^n c_k z^k$, if we look at its restriction to the unit circle and write each $z$ of modulus one as $z=e^{i t}$, $t\in [0,2\pi]$, it is easy to see that Re\,$P$ is a trigonometric polynomial of degree $n$:
\begin{equation}
 T(t)=\a_0+\sum_{k=1}^n (\a_k \cos k t + \b_k \sin k t)	\,.
 \label{trig-pol}
\end{equation}
It is not difficult to see that it can have at most $2n$ zeros in $[0,2\pi]$. Either from this or from the solution to the Dirichlet problem for the disk with continuous data on the unit circle, we deduce  the following.
\smallskip\par\noindent
{\sc Fact}.
\textit{If the real part of a complex polynomial $P$ vanishes on the unit circle then $P$ is identically equal to a purely imaginary constant.}
\par\medskip
\textbf{The Fej\'er lemma}. The following classical lemma due to Fej\'er \cite[p.~154--155]{Ts} characterizes an important class of
trigonometric polynomials.
\smallskip\par\noindent
{\sc Fej\'er's Lemma}. \textit{If $T$ is a trigonometric polynomial as in \eqref{trig-pol} and $T(t)\ge 0$ for all $t\in [0,2\pi]$ then there exist complex coefficients $\g_j$, $0\le j\le n$, such that}
$$
 T(t)=|\g_0+\g_1 e^{i t} +\ldots + \g_n e^{i n t}|^2\,, \quad t\in [0,2\pi]\,.
$$
\par\medskip
\textbf{Another useful property}. Here is a useful fact. The argument is adapted from the proof of the main theorem in our recent work on a different topic \cite{MSUV}.
\begin{prop}
 \label{prop-pos}
If $n\ge 2$ and the polynomial
$$
 Q(z)=1+c_1 z+c_2 z^2 + \ldots + c_{n-1} z^{n-1} + z^n
$$ has positive real part in \,$\D$ then $c_1=c_2=\ldots=c_{n-1}=0$.
\end{prop}
\begin{proof}
Consider the $n$-th roots of $-1$:
$$
 \om_k = e^{(2k+1) \pi i / n}\,, \quad k= 0, 1,\ldots, n-1\,.
$$
Clearly, $\mathrm{Re\,} \{1 + \om_k^n\} = 0$, $k= 0, 1,\ldots, n-1$. Hence from our assumption that $Q$ has positive real part in $\D$ we conclude that for each of these values
\begin{equation}
 \mathrm{Re\,} \{c_{1} \om_k + c_{2} \om_k^2 + \ldots + c_{n-1} \om_k^{n-1}\} \ge 0\,, \quad k= 0, 1,\ldots, n-1\,.
 \label{coeff-simple}
\end{equation}
By basic algebra, for any fixed $j$ with $1\le j\le n-1$ we have
$$
 \sum_{k=0}^{n-1} \om_k^j = e^{\pi i j/ n} \sum_{k=0}^{n-1} e^{2 k j \pi i / n}	
 = e^{\pi i j/ n} \frac{1 - e^{2 j \pi i}}{1 - e^{2 j \pi i / n}} = 0\,.
$$
Thus, summing up the terms on the left in \eqref{coeff-simple} over $k\in\{0,1,\ldots,n-1\}$, we get
$$
 \sum_{k=0}^{n-1} \mathrm{Re\,} \{c_{1} \om_k + c_{2} \om_k^2 + \ldots + c_{n-1} \om_k^{n-1}\} = \mathrm{Re\,} \sum_{j=1}^{n-1} \left\{c_{j} \sum_{k=0}^{n-1} \om_k^j \right\} = 0\,.
$$
Since every summand on the left-hand side in the above formula is non-negative in view of \eqref{coeff-simple}, all of them must be zero:
$$
 \mathrm{Re\,} \{c_1 \om_k + c_{2} \om_k^2 + \ldots + c_{n-1} \om_k^{n-1}\} = 0\,, \qquad k = 0, 1,\ldots, n-1\,,
$$
hence also
\begin{equation}
 \mathrm{Re\,} \{1+ c_1 \om_k + c_{2} \om_k^2 + \ldots + c_{n-1} \om_k^{n-1} + \om_k^n \} = 0\,, \qquad k= 0, 1,\ldots, n-1\,.
 \label{all-zero}
\end{equation}
Writing $\la=e^{i t}$, $t\in [0,2\pi]$, the function
$$
 T(t) = \mathrm{Re\,} \{1+c_{1}\la+ \ldots +c_{n-1}\la^{n-1} +\la^n\}
$$
can be viewed as a trigonometric polynomial of degree $n$ of the variable $t$. Since $T(t)\ge 0$ on $[0,2\pi]$, Fej\'er's Lemma tells us that for some coefficients $\g_0$,$\g_1$,\ldots,$\g_n$ and $\la=e^{i t}$ we have
$$
 T(t)=|(\g_0+\g_1 \la +\ldots + \g_n \la^n)^2|\,.
$$
The complex polynomial $Q(z)=(\g_0+\g_1 z +\ldots \g_n z^n)^2$ cannot be identically zero for then the trigonometric polynomial $T$ would be identically zero in $[0,2\pi]$ and then the restriction of $\mathrm{Re\,} \{1+c_{1} z+ \ldots +c_{n-1} z^{n-1} + z^n\}$ to the unit circle would be zero while its value at the origin is one, which would contradict the maximum principle or the mean value property. Therefore the polynomial $Q$ has $2n$ zeros counting the multiplicities, each zero being obviously of order
at least two. But we know from \eqref{all-zero} that this polynomial
has at least $n$ distinct zeros $\om_k$, $k= 0, 1,\ldots, n-1$, which
are roots of $-1$, so each one of these zeros must be double and hence
$Q$ cannot have any other zeros. Thus, the polynomial factorizes as
$$
 Q(z) = (\g_0+\g_1 z +\ldots + \g_n z^n)^2 = C \prod_{k=0}^{n-1} (z -
 \om_k)^2 = C (z^n+1)^2 \,.
$$
Hence
$$
 \mathrm{Re\,} \{1 + c_{1} z + \ldots + c_{n-1} z^{n-1} +
 z^n\} = |C (z^n + 1)^2| = 2 |C|\,\mathrm{Re\,} \{z^n + 1\}
$$
for all $z$ on the unit circle. From the fact quoted earlier, two
polynomials whose real parts are equal on the unit circle must coincide everywhere, except for an imaginary constant:
$$
 1 + c_{1} z + \ldots + c_{n-1} z^{n-1} + z^n = 2 |C|
 (z^n+1) + ic\,, \qquad z\in\C\,, \quad c\in\R\,.
$$
It follows that
$$
 c_1=c_2=\ldots=c_{n-1}=0\,,
$$
which proves the claim.
\end{proof}
\par\medskip
\textbf{A remark on simple critical points}. It will also be important to stress that if all critical points of a polynomial univalent in $\D$ are on the unit circle, then all of them are simple zeros of the derivative. This fact is known to the experts \cite[p.~105]{Br}, \cite[p.~241]{SS} but it seems useful to explain it in a few lines. In fact, if $P^\prime(\a)=0$ for some $\a$ with $|\a|=1$ and if $P^{\prime\prime} (\a)=0$, then $\a$ is a zero of order $k\ge 3$ of the polynomial $P(z)-P(\a)$, hence
$$
 P_\a (z) = P(z)-P(\a)= [(z-\a) h(z)]^k
$$
for some entire function $h$. Note that the function $g(z)=(z-\a) h(z)$ has the property that $g^\prime(\a)=h(\a)\neq 0$, hence $g$ is conformal at the point $\a$. Now the basic local mapping properties can be used to contradict the univalence of $P_\a$.
\section{The main result and its proof}
 \label{sec-main}
\par
We are now ready to prove our main result. It can be viewed as an extension of Brannan's result by a completely different and possibly simpler method. In order to prove the equivalence between (b) and (c) below in \cite{Br}, Brannan used Dieudonn\'e's criterion and another lemma on univalent polynomials proved by himself in an earlier paper. As far as we know, the equivalence between (a) and the remaining two conditions is new.
\begin{thm}\label{gen-brennan}
Let $P(z)=z+a_2 z^2+\ldots + a_{n-1}z^{n-1}+a_n z^n$ be a polynomial in the class $S$ with all critical points on the unit circle (that is, $|a_n|=\frac{1}{n}$). Then the following statements are equivalent:
\begin{itemize}
 \item[(a)]
$P^\prime$ has positive real part in $\D$.
 \item[(b)]
$a_2=\ldots=a_{n-1}=0$.
 \item[(c)]
$P$ is starlike.
\end{itemize}
\end{thm}
\par\smallskip
Note that $P$ has any of the properties (a)--(c) if and only if any of its rotations $P_\la(z)=\overline{\la} P(\la z)$, $|\la|=1$, has the corresponding property. Thus, without loss of generality we may assume that $P(z)=z+a_2 z^2+\ldots + a_{n-1}z^{n-1}+\frac{z^n}{n}$. We proceed under this assumption.
\par\smallskip
We first show that (a)$\Leftrightarrow$(b) and then also that (b)$\Leftrightarrow$(c).
\par\smallskip
\begin{proof}
\fbox{(a)$\Rightarrow$(b)} \ Follows directly from Proposition~\ref{prop-pos} since
$$
 P^\prime (z)=1+2a_2 z+\ldots + (n-1)a_{n-1}z^{n-2}+z^{n-1}\,.
$$
\par\smallskip
\fbox{(b)$\Rightarrow$(a)} \ In this case, $P^\prime (z)=1+z^{n-1}$, and the conclusion is clear.
\par\smallskip
\fbox{(b)$\Rightarrow$(c)} \ It is straightforward to check that
$$
 Q(z) = \frac{z P^\prime(z)}{P(z)} = \frac{1+z^{n-1}}{1+ \frac{z^{n-1}}{n}}
$$
has positive real part in $\D$ since the linear fractional mapping  $\frac{1+w}{1+\frac{w}{n}}$ maps the unit disk onto a disk whose  diameter is $(0,\frac{2n}{n+1})$.
\par\smallskip
\fbox{(c)$\Rightarrow$(b)} \ By the well-known criterion for starlikeness, $P$ must have the property that
$$
 \mathrm{Re\,} \left\{ \frac{z P^\prime(z)}{P(z)} \right\} > 0 \quad
 \mathrm{for \ all \ } z\in\D\,.
$$
After multiplying both sides by $|P(z)|^2$, it follows that
$$
 \mathrm{Re\,} \left\{ z P^\prime(z) \overline{P(z)} \right\} \ge 0
$$
for all $z\in\D$ and hence also for $z\in\overline{\D}$. Thus, if we denote by $R$ the polynomial defined by
$$
 R(z) = \frac{P(z)}{z} = 1+a_2 z + \ldots + a_{n-1}z^{n-2} +\frac{z^{n-1}}{n}\,,
$$
we conclude that
$$
 \mathrm{Re\,} \left\{ P^\prime(z) \overline{R(z)} \right\} \ge 0\,,
$$
for all $z$ on the unit circle $\T=\partial\D$. Since the restriction of $\mathrm{Re\,} \{P^\prime \overline{R}\}$ to the unit circle is a trigonometric polynomial, the Fej\'er lemma yields
$$
 \mathrm{Re\,} \left\{ P^\prime(z) \overline{R(z)} \right\} = |Q(z)|^2  \quad  \mathrm{for \ all \ } z\in\T\,,
$$
where $Q(z)=\g_0+\g_1 z+\ldots +\g_{n-1} z^{n-1}$. Now, the zeros of $P^\prime$ all lie on $\T$ by assumption and are pairwise different (as observed earlier), so there are $n-1$ of them. On the other hand, $Q$ viewed as a polynomial in the complex plane must vanish at each zero of $P^\prime$ and has $n-1$  zeros. Hence the zeros of $Q$  must all coincide with the zeros of $P^\prime$ on $\T$, and it follows that actually in the whole plane we have $Q(z) = C P^\prime (z)$ for some constant $C$. Hence
$$
 \mathrm{Re\,} \left\{ P^\prime(z) \overline{R(z)} \right\} = |C|^2  |P^\prime (z)|^2 \quad  \mathrm{for \ all \ } z\in\T\,.
$$
In other words,
$$
 \mathrm{Re\,} \left\{ P^\prime(z) \overline{(R(z) - |C|^2 P^\prime (z))} \right\} = 0 \quad  \mathrm{for \ all \ } z\in\T\,.
$$
After writing down the expressions for both factors on the left-hand side, $P ^\prime$ and $\overline{(R - |C|^2 P^\prime)}$, and multiplying out, we obtain a trigonometric polynomial of the form
$$
 \a_0+\sum_{k=1}^{n-1} (\a_k \cos k t + \b_k \sin k t)
$$
whose all coefficients are zero. There is no need to compute all of them: it suffices to focus just on $\a_0$ and $\a_{n-1}$. One easily notices that
\begin{equation}
 \a_{n-1} = 1+\frac{1}{n}-2|C|^2=0\,.
 \label{eq-const}
\end{equation}
and
$$
 \a_0 = (1-|C|^2) + \sum_{k=2}^{n-1} k (1-k|C|^2) |a_k|^2 + \(\frac{1}{n}-|C|^2\) = 0\,.
$$
In view of \eqref{eq-const}, the first and the last term in the above sum cancel out, so we are left with
\begin{equation}
 \sum_{k=2}^{n-1} k \(1-k|C|^2\) |a_k|^2 = 0\,.
 \label{eq-var}
\end{equation}
Equation \eqref{eq-const} yields $|C|^2=(n+1)/(2n)$, which easily implies that
$$
 1-k|C|^2<0, \qquad \mathrm{for \ all \ } \ k\in\{2,3,\ldots,n-1\}\,.
$$
This, together with \eqref{eq-var}, readily implies (b).
\end{proof}
\par


\end{document}